\documentclass[11pt]{article}
\usepackage{epsfig}
\usepackage{graphicx}
\usepackage{amssymb,amsmath}
\usepackage{caption}
\usepackage{cancel}
\usepackage{tikz}
\newtheorem{thm}{{\bf Theorem}}[section]

\newtheorem{lemma}[thm]{\bf Lemma}

\newtheorem{pro}[thm]{\bf Proposition}

\newtheorem{definition}[thm]{\bf Definition}
\newtheorem{preproof}{{\bf Proof.}}

\newenvironment{proof}[1]{\begin{preproof}{\rm
               #1}\hfill{$\rule{2mm}{2mm}$}}{\end{preproof}}
\begin{document}
\title{\Large {\bf A characterization of some graphs with metric dimension two}}

\author{
{\sc Ali Behtoei$^a$\thanks{a.behtoei@sci.ikiu.ac.ir}}, {\sc Akbar Davoodi$^b$\thanks{a.davoodi@math.iut.ac.ir}}, {\sc Mohsen Jannesari$^c$\thanks{mjannesari@shahreza.ac.ir}} and
{\sc Behnaz Omoomi$^d$\thanks{bomoomi@cc.iut.ac.ir}}\\
[1mm]
{$^a$\it \small Department of Mathematics, Imam Khomeini International University, 34149-16818, Qazvin, Iran
}\\
{$^{b,d}$ \it \small Department of Mathematical Sciences, Isfahan University of Technology, 84156-83111, Isfahan, Iran}\\
{$^c$ \it\small University of Shahreza, 86149-56841, Shahreza, Iran}}
\date{}

\maketitle

\begin{abstract}
 A set $W\subseteq V(G)$ is called a  resolving set, if
for each pair of distinct vertices $u,v\in V(G)$ there exists $t\in W$
such that $d(u,t)\neq d(v,t)$,  where $d(x,y)$ is the distance
between  vertices $x$ and $y$. The cardinality of a minimum
resolving set for $G$ is called the  metric dimension of $G$ and
is denoted by $\dim_M(G)$.
A $k$-tree is a chordal graph all of whose maximal cliques are the same size $k+1$
 and all of whose minimal clique separators are also all the same size $k$.
A $k$-path is a $k$-tree with maximum degree $2k$, where for each integer  $j$,  $k\leq j<2k$, there exists a unique pair of vertices, $u$ and $v$, such that $\deg(u)=\deg(v)=j$.
In this paper, we prove that if $G$ is a $k$-path, then $\dim_M(G)=k$.
Moreover, we provide a characterization of all $2$-trees with metric dimension two.
\end{abstract}

\section{Introduction}
Throughout this paper all graphs are finite, simple and undirected. The notions $\delta$, $\Delta$ and $N_G(v)$ stand for minimum degree, maximum degree and the set of neighbours of vertex $v$ in $G$, respectively.

For an ordered set $W=\{w_1,w_2,\ldots,w_k\}$ of vertices and a
vertex $v$ in a connected graph $G$, the $k$-vector
$r(v|W):=(d(v,w_1),d(v,w_2),\ldots,d(v,w_k))$ is  called  the
\textit{metric representation} of $v$ with respect to $W$, where $d(x,y)$
is the distance between two vertices $x$ and $y$. The set $W$ is
called  a \textit{resolving set} for $G$ if distinct vertices of $G$ have
distinct representations with respect to $W$. 
We say a set $S\subseteq V(G)$ \textit{resolves} a set $T\subseteq V(G)$ if for each pair of distinct vertices $u$ and $v$ in $T$ there is a vertex $s\in S$ such that $d(u,s)\neq d(v,s)$. A  minimum resolving set is called a \textit{basis}  and the \textit{metric dimension} of $G$, $\dim_M(G)$, is
the cardinality of a basis  for $G$. A graph with metric dimension $k$
 is called $k$-\textit{dimensional}.

The concept of the resolving set has various applications in diverse areas
including coin weighing problems~\cite{coin},
 network discovery and verification~\cite{net2},
  robot navigation~\cite{landmarks},
mastermind game~\cite{cartesian product},
problems of pattern recognition and image processing~\cite{digital},
 and combinatorial search and optimization~\cite{coin}.

These concepts were introduced by  Slater in~\cite{Slater1975}.
 He described the usefulness of these concepts when
working with U.S. Sonar and Coast Guard Loran stations.
Independently, Harary and Melter~\cite{Harary} discovered these
concepts. In \cite{landmarks}, it is proved that  determining the metric dimension of a graph in general is an $NP$-complete problem, but the metric dimension of trees can be  obtained  by a polynomial time algorithm.

 It is obvious that for every graph $G$ of order $n$, $1\leq \dim_M(G) \leq  n-1$.
  Chartrand et al.~\cite{Ollerman} proved that for $n\geq 2$, $\dim_M(G)=n-1$
if and only if $G$ is the complete graph $K_n$. They also provided a
 characterization of  graphs of order $n$ and
metric dimension $n-2$~\cite{Ollerman}.  Graphs with metric dimension $n-3$ are characterized in~\cite{n-3}.
  Khuller et al.~\cite{landmarks} and Chartrand et al.~\cite{Ollerman}
  proved that $\dim_M(G)=1$ if and only if $G$ is a path.
Moreover,  in~\cite{chang}  some properties of $2$-dimensional graphs are obtained.
\begin{thm}\label{thm:degree of basis elements}{\em\cite{chang}}
Let $G$ be a $2$-dimensional graph. If  $\{a,b\}$ is a basis for $G$, then
\begin{enumerate}
\item
there is a unique shortest path $P$ between $a$ and $b$, 
\item
the degrees of $a$ and $b$ are at most three, 
\item
 the degree of each internal vertex  on $P$ is  at most five.
\end{enumerate}
\end{thm}

A \textit{chordal graph} is a graph with no induced cycle of length greater than three.
A \textit{$k$-tree} is a chordal graph that  all of whose maximal cliques are the same size $k+1$
and all of whose minimal clique separators are also all the same size $k$.
In other words, a $k$-tree may be formed by starting with a set of $k+1$ pairwise adjacent vertices and then repeatedly adding vertices in such a way that each added vertex has exactly $k$ neighbours that form a $k$-clique. 

By the above definition, it is clear that if $G$ is a $k$-tree, then $\delta(G)=k$.
$1$-trees are the same as trees; $2$-trees are maximal series-parallel graphs~\cite{2tree} and include also the maximal outer-planar graphs. These graphs can be used to model series and parallel electric circuits.
 Planar $3$-trees are also known as Apollonian networks~\cite{3-tree}.

A \textit{$k$-path} is a $k$-tree with maximum degree  $2k$, where for each integer $j$,  $k\leq j<2k$, there exists a unique pair of vertices, $u$ and $v$, such that $\deg(u)=\deg(v)=j$.
On the other hand, regards to the recursive construction of $k$-trees, a $k$-path $G$ can be considered as a graph with vertex set $ V(G)=\{v_1,v_2,\ldots,v_n\}$ and edge set $ E(G)=\{v_iv_j:~ |i-j|\leq k\}.$  For instance, two different representations of a $2$-path $G$ with seven vertices $v_1,\ldots, v_7$ are shown in Figure \ref{fig:twoRep}.
\begin{center}
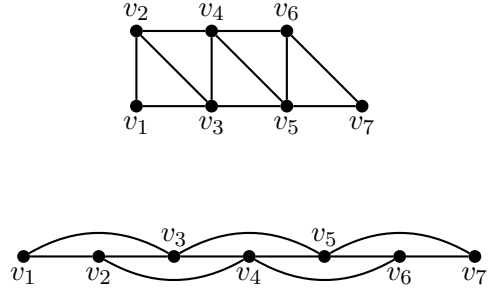

\begin{tikzpicture}
[inner sep=0.5mm, place/.style={circle,draw=black,fill=black,thick}]
\node[place] (v1) at (-2.5,.5) [label=below:$v_1$] {};
\node[place] (v2) at (-2.5,1.5) [label=above:$v_2$] {}edge [-,thick](v1);
\node[place] (v3) at (-1.5,.5) [label=below:$v_3$] {}edge [-,thick](v1)edge [-,thick](v2);
\node[place] (v4) at (-1.5,1.5) [label=above:$v_4$] {}edge [-,thick](v2)edge [-,thick](v3);
\node[place] (v5) at (-.5,.5) [label=below:$v_5$] {}edge [-,thick](v3)edge [-,thick](v4);
\node[place] (v6) at (-.5,1.5) [label=above:$v_6$] {}edge [-,thick](v4)edge [-,thick](v5);
\node[place] (v7) at (.5,.5) [label=below:$v_7$] {}edge [-,thick](v5)edge [-,thick](v6);
\node[place] (v1') at (-4,-1.5) [label=below:$v_1$] {};
\node[place] (v2') at (-3,-1.5) [label=below:$v_2$] {}edge [-,thick](v1');
\node[place] (v3') at (-2,-1.5) [label=above:$v_3$] {}edge [-,thick, bend right](v1')edge [-,thick](v2');
\node[place] (v4') at (-1,-1.5) [label=below:$v_4$] {}edge [-,thick,bend left](v2')edge [-,thick](v3');
\node[place] (v5') at (0,-1.5) [label=above:$v_5$] {}edge [-,thick,bend right](v3')edge [-,thick](v4');
\node[place] (v6') at (1,-1.5) [label=below:$v_6$] {}edge [-,thick,bend left](v4')edge [-,thick](v5');
\node[place] (v7') at (2,-1.5) [label=below:$v_7$] {}edge [-,thick,bend right](v5')edge [-,thick](v6');
\end{tikzpicture}
\captionof{figure}{Two different representations of  a $2$-path.\label{fig:twoRep}}
\end{center}

In this paper, we show that the metric dimension of  each $k$-path (as a generalization of a path) is $k$.
Whereas, there are some examples of $2$-trees with metric dimension two that are not $2$-path.
 This fact motivates us to study the structure of $2$-dimensional $2$-trees.
As a  main result, we characterize the class of all $2$-trees with metric dimension two.

\section{Main Results}
In this section, we first prove that the metric dimension of each $k$-path is $k$. Then, we introduce a class of graphs which shows that the inverse of this fact is not true in general. Later on, we concern on the case $k=2$ and toward to investigating all $2$-trees with metric dimension two, we construct a family $\cal F$ of $2$-trees with metric dimension two. 
Finally, as the main result, we prove that the metric dimension of a $2$-tree $G$ is two if and only if $G$ belongs to  $\cal F$.

\begin{thm}\label{k-path}
If $G$ is a $k$-path, then $\dim_{_M}(G)=k$.
\end{thm}

\begin{proof}{
Let $G$ be a $k$-path with vertex set $ V(G)=\{v_1,v_2,\ldots,v_n\}$ and edge set $ E(G)=\{v_iv_j:~ |i-j|\leq k\}$.
Therefore, the distance between two vertices $v_r$ and $v_s$ in $G$ is given by $d(v_r,v_s)=\left\lceil {|r-s|\over k}\right\rceil$.
 	
 	At first, let $W=\{v_1,v_2,\ldots,v_k\}$ and $v_i$, $v_j$ be two distinct vertices of $G$ with $k<i<j$. By the division algorithm, there exist integers $r$ and $s$ such that $i=rk+s$, $1\leq s\leq k$. Thus, we have
 $$d(v_i,v_s)=\left\lceil{|i-s|\over k}\right\rceil=\left\lceil{rk\over k}\right\rceil=r,$$
 and
 $$d(v_j,v_s)=\left\lceil{|j-s|\over k}\right\rceil=\left\lceil{rk+(j-i)\over k}\right\rceil=
 r+\left\lceil{j-i\over k}\right\rceil\geq r+1.$$
  This means $W$ is a resolving set for $G$. Hence, $\dim_M(G)\leq |W|=k$.
 
 Now, we show that $\dim_M(G)\geq k$.
 Let $W$ be a  basis of the $k$-path $G$, and let $X=\{v_1,v_2,\ldots,v_{k+1}\}$.
Assume that $|W\cap X|=s$ and $X\setminus W=\{v_{i_1},v_{i_2},\ldots,v_{i_{k+1-s}}\}$, where $1\leq i_1<i_2<\cdots<i_{k+1-s}\leq k+1$. For convince,  let $X'=\{x_1,x_2,\ldots,x_{k+1-s}\}$, where $x_r=v_{i_r}$, for each $r$, $1\leq r\leq k+1-s$.
Since each vertex $v_i$ of the $k$-path $G$ is adjacent to the next $k$ consecutive vertices $\{v_{i+1},\ldots, v_{i+k}\}$, the induced subgraph on $X$ is a $(k+1)$-clique.
Each vertex in $W\cap X$ is adjacent to each vertex in $X'$. Thus, each pair of vertices in $X'$ should be resolved by some element of $W\setminus X$. Assume that $W'=\{w_1,w_2,\ldots,w_t\}$ is a minimum subset of $W\setminus X$ which resolves vertices in $X'$.
Thus, for each $w_j\in W'$ there exists $\{x_r,x_s\}\subseteq X'$ such that $d(w_j,x_r)\neq d(w_j,x_s)$.
For each $j$, $1\leq j\leq t$, let
$$r_j=\min\{r:~d(w_j,x_r)\neq d(w_j,x_{r+1})\},$$
and, let $$A_j=\{x_1,x_2,\ldots,x_{r_j}\},~B_j=\{x_{r_j+1},x_{r_j+2},\ldots,x_{k+1-s}\}.$$
Note that $A_j\cup B_j=X'$, $A_j\cap B_j=\emptyset$, $x_1\in A_j$ and $x_{k+1-s}\in B_j$.
Also, the structure of $G$ implies that
$$d(w_j,x_1)=d(w_j,x_2)=\cdots=d(w_j,x_{r_j}),$$
and $$d(w_j,x_{r_j+1})=d(w_j,x_{r_j+2})=\cdots=d(w_j,x_{k+1-s}).$$
Since $W'$ has the minimum size, for each $1\leq j<j'\leq t$ we have $A_{j}\neq A_{j'}$ (otherwise, $w_{j}$ and $w_{j'}$ resolve the same pair of vertices in $X'$) and hence, $|A_{j}|\neq |A_{j'}|$. Moreover,  
 for each $r$, $1\leq r\leq k-s$, there exists $w_j\in W'$ such that $d(w_j,x_r)\neq d(w_j,x_{r+1})$ which implies $|A_j|=r$. 
Therefore, $$t=\left|\{|A_1|,|A_2|,\ldots,|A_t|\}\right|=|\{1,2,\ldots,k-s\}|=k-s.$$
Hence, $$|W|=|W\setminus X|+|W\cap X|\geq |W'|+s=(k-s)+s=k,$$
which completes the proof.
 }\end{proof}
 
\begin{definition}
Let $G$ and $H$ be two $2$-trees. We say that  $H$ is a \textit{branch} in $G$ on $\{u,v\}$, for convenience say a $(u,v)$-branch,  if $V(H)\cap V(G)=\{u,v\}$, where $uv$ is an edge of $G$ belonging to only one of the triangles in $H$.
The \textit{length} of a branch in a $2$-tree is the number of it's triangles, which is equal to the number of vertices of branch minus $2$. A cane is a $2$-path with a branch of length one on a specific edge as shown in Figure \ref{cane}.
\begin{center}
\begin{tikzpicture}
[rotate=90,inner sep=0.5mm, place/.style={circle,draw=black,fill=black,thick}]
\node[place] (r1) at (-2,4)  {};
\node[place] (r2) at (-2,3)  {}edge [-,thick](r1);
\node[place] (r3) at (-2,2) {}edge [-,thick](r2);
\node[place] (rk-1) at (-2,1) {}edge [-,thick](r3);
\node[place] (s1) at (-1,4)  {}edge [-,thick](r1);
\node[place] (s2) at (-1,3) {}edge [-,thick](r2)edge [-,thick](s1)edge [-,thick](r1);
\node[place] (s3) at (-1,2) {}edge [-,thick](r3)edge [-,thick](s2)edge [-,thick](r2);
\node[place] (sk-1) at (-1,1) {}edge [-,thick](rk-1)edge [-,thick](s3)edge [-,thick](r3);
\node (dots) at (-1.5,.5) [label=center:${\cdots}$]{};
\node[place] (c2) at (0,4)   {}edge [-,thick](s1)edge [-,thick](s2);
\node[place] (r11) at (-2,0) 
{};
\node[place] (r22) at (-2,-1) {}edge [-,thick](r11);
\node[place] (s11) at (-1,0) 
{}edge [-,thick](r11);
\node[place] (s22) at (-1,-1)  {}edge [-,thick](r22)edge [-,thick](s11)edge [-,thick](r11);
%
\end{tikzpicture}
\captionof{figure}{A cane.\label{cane}}
\end{center}

\end{definition}
 
 In the following proposition,  we provide some $2$-trees with metric dimension two other than $2$-paths. 
\begin{pro}\label{lem:d=1}
If $G$ is a $2$-tree of metric dimension two with a basis whose elements are adjacent, then $G$ is a $2$-path or a cane.
\begin{center}
\begin{tikzpicture}
[inner sep=0.5mm, place/.style={circle,draw=black,fill=black,thick}]
\node[place] (r1) at (-6,4) [label=above:$a$] {};
\node[place] (r2) at (-6,3) [label=left:$(1\text{,}2)$] {}edge [-,thick](r1);
\node[place] (r3) at (-6,2) {}edge [-,thick](r2);
\node[place] (rk-1) at (-6,1) {}edge [-,thick](r3);
\node[place] (s1) at (-5,4) [label=above:$b$] {}edge [-,thick](r1);
\node[place] (s2) at (-5,3) [label=right:$(1\text{,}1)$]{}edge [-,thick](r2)edge [-,thick](s1)edge [-,thick](r1);
\node[place] (s3) at (-5,2) {}edge [-,thick](r3)edge [-,thick](s2)edge [-,thick](r2);
\node[place] (sk-1) at (-5,1) {}edge [-,thick](rk-1)edge [-,thick](s3)edge [-,thick](r3);
\node (dots) at (-5.5,.5) [label=center:${\vdots}$]{};
\node[place] (c2) at (-4,4)  [label=above:$(2\text{,}1)$] {}edge [-,thick](s1)edge [-,thick](s2);
\node[place] (r11) at (-6,0) 
{};
\node[place] (r22) at (-6,-1) [label=left:$(t\text{,}t+1)$] {}edge [dashed,thick](r11);
\node[place] (s11) at (-5,0) 
{}edge [-,thick](r11);
\node[place] (s22) at (-5,-1) [label=right:$(t\text{,}t)$] {}edge [dashed,thick](r22)edge [-,thick](s11)edge [-,thick](r11);
\node (dots) at (-5.5,-2.5) [label=center:{\em(a)}]{};
\node[place] (a1) at (-2,4) {};
\node[place,line width=3pt] (a2) at (-2,3) {}edge [-,thick](a1);
\node[place] (a3) at (-2,2) {}edge [-,thick](a2);
\node[place] (ak-1) at (-2,1) {}edge [-,thick](a3);
\node[place] (b1) at (-1,4) [label=above:$a$] {}edge [-,thick](a1);
\node[place] (b2) at (-1,3) {}edge [-,thick](a2)edge [-,thick](b1)edge [-,thick](a1);
\node[place,line width=3pt] (b3) at (-1,2) {}edge [-,thick](a3)edge [-,thick](b2)edge [-,thick](a2);
\node[place] (bk-1) at (-1,1) {}edge [-,thick](ak-1)edge [-,thick](b3)edge [-,thick](a3);
\node (dots) at (-1.5,.5) [label=center:${\vdots}$]{};
\node[place] (r11) at (-2,0) {};
\node[place] (r22) at (-2,-1) {}edge [dashed,thick](r11);
\node[place] (s11) at (-1,0) {}edge [-,thick](r11);
\node[place] (s22) at (-1,-1) {}edge [dashed,thick](r22)edge [-,thick](s11)edge [-,thick](r11);
\node (dots) at (-1.5,-2.5) [label=center:{\em(b)}]{};
\node[place] (c1) at (0,4) [label=above:$b$] {}edge [-,thick](b1)edge [-,thick](b2);
\node[place] (u1) at (2,4) {};
\node[place] (u2) at (2,3) {}edge [-,thick](u1);
\node[place] (u3) at (2,2) {}edge [-,thick](u2);
\node[place] (uk-1) at (2,1) {}edge [-,thick](u3);
\node[place,line width=3pt] (v1) at (3,4) {}edge [-,thick](u1);
\node[place] (v2) at (3,3) {}edge [-,thick](u2)edge [-,thick](v1)edge [-,thick](u1);
\node[place] (v3) at (3,2) {}edge [-,thick](u3)edge [-,thick](v2)edge [-,thick](u2);
\node[place] (vk-1) at (3,1) {}edge [-,thick](uk-1)edge [-,thick](v3)edge [-,thick](u3);
\node (dots) at (2.5,.5) [label=center:${\vdots}$]{};
\node[place,line width=3pt] (c3) at (4,4) {}edge [-,thick](v1)edge [-,thick](v2);
\node[place] (u11) at (2,0){};
\node[place] (u22) at (2,-1)  [label=below:$a$]{}edge [dashed,thick](u11);
\node[place] (v11) at (3,0)  {}edge [-,thick](u11);
\node[place] (v22) at (3,-1) [label=below:$b$] {}edge [dashed,thick](u22)edge [-,thick](v11)edge [-,thick](u11);
\node (dots) at (2.5,-2.5) [label=center:{\em(c)}]{};
\node[place] (x1) at (6,4) [label=above:$a$] {};
\node[place] (x2) at (6,3) [label=left:$(1\text{,}2)$]{}edge [-,thick](x1);
\node[place] (x3) at (6,2) {}edge [-,thick](x2);
\node[place] (xk-1) at (6,1) {}edge [-,thick](x3);
\node[place] (y1) at (7,4) [label=above:$b$] {}edge [-,thick](x1);
\node[place] (y2) at (7,3)  [label=right:$(1\text{,}1)$]{}edge [-,thick](x2)edge [-,thick](y1)edge [-,thick](x1);
\node[place] (y3) at (7,2) {}edge [-,thick](x3)edge [-,thick](y2)edge [-,thick](x2);
\node[place] (yk-1) at (7,1) {}edge [-,thick](xk-1)edge [-,thick](y3)edge [-,thick](x3);
\node (dots) at (6.5,.5) [label=center:${\vdots}$]{};
\node[place] (r11) at (6,0) {};
\node[place] (r22) at (6,-1) [label=left:$(t\text{,}t+1)$]{}edge [dashed,thick](r11);
\node[place] (s11) at (7,0) {}edge [-,thick](r11);
\node[place] (s22) at (7,-1)[label=right:$(t\text{,}t)$] {}edge [dashed,thick](r22)edge [-,thick](s11)edge [-,thick](r11);

\node (dots) at (6.5,-2.5) [label=center:{\em(d)}]{};
\end{tikzpicture}

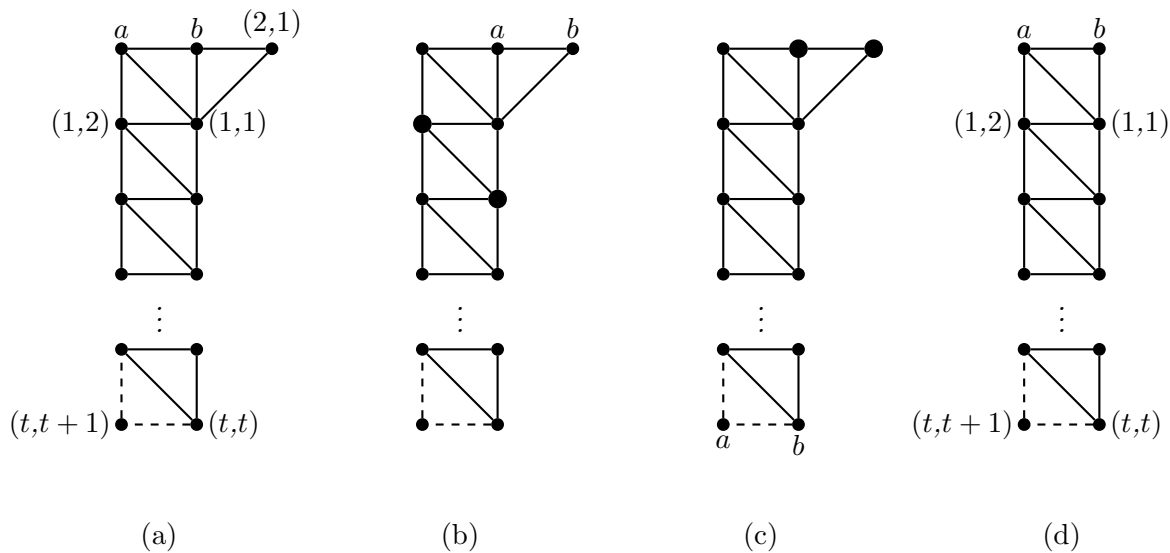
\captionof{figure}{The possible cases for basis $\{a,b\}$ in $2$-tree $G$\label{cane2}}
\end{center}

\end{pro}
\begin{proof}
{We prove the statement by induction on $n$, the order of $G$. If $n=3$, then $G=K_3$ and the statement holds.
 Let $G$ be a $2$-tree of order $n >3$ with a basis $
B=\{a,b\}$, such that $d(a,b)=1$. Since each $2$-tree of order greater than three has two non-adjacent vertices of degree two, there exists a vertex $x\in
V(G)\setminus B$ of degree two. Moreover,  $B$ is a basis for
$G\setminus\{x\}$.

Now, by the induction hypothesis,
$G\setminus\{x\}$ is a path or a cane and by Theorem~\ref{thm:degree of basis elements} (2), the degrees of $a$ and $b$ are at most three. Therefore, $B=\{a,b\}$ is one of the possible cases shown in Figure~\ref{cane2}. Note that dashed edges could be absent. It can be checked that in cases (b) and (c) the  bold vertices get the same metric representation with respect to $B$.
Thus, $B$ is one of the cases (a) or (d), where the metric representations of vertices are denoted in Figure~\ref{cane2}.

Regards to the metric representation of vertices in $G$,  $x$ could be adjacent to the vertices by metric representation $(t, t+1)$ and $(t,t)$ (in the case of not existence of dashed edges $(t-1,t)$ and $(t,t)$) and in the case (d) to the vertices by metric representation $(1,0)$ and $(1,1)$ as well. This concludes that $G$ is also a path or a cane.
}
\end{proof}

The above proposition shows that the inverse of Theorem~\ref{k-path} is not true.
Later on, we focus on the case $k=2$ and construct the family $\mathcal{F}$ of all $2$-trees with metric dimension two.

Let $\mathcal {F}$ be the family of $2$-trees, where each member $G$ of ${\cal F}$ consists of a $2$-tree $G_0$ and some branches on it that,  in the case of existence,  satisfying the following conditions.
\begin{enumerate}
\item
$G_0$ is a $2$-path or  a $2$-tree that is obtained by identifying two specific edges of two disjoint $2$-paths as shown in Figure \ref{fig:2-path P}.
\item
On every edge there is at most one branch.
 \item
$G$ avoids any $(a_i,a_{i+1})$-branch.
\item
Each branch is either a $2$-path or a cane.
\item
In each $(a_i,b_{i})$-branch the degree of $a_i$ is two.
\item
If $G_0$ is as the graph depicted in Figure~\ref{fig:2-path P}($b$), then $G$ avoids any $(a_m, x)$-branch.
\item
 $G$ contains at most one branch on the edges of
the triangle containing $b_ib_{i+1}$ in $G_0$.
\item
The degree of each $b_i$ in $G$ is at most $7$.
\item
 $G$ has at most one branch of length greater than one on the edges of
the triangle containing $a_ia_{i+1}$ in $G_0$.
\item
If $G_0$ is of the form of  Figure \ref{fig:2-path P}($b$), then $(b_{m-1},b_m)$-branch and $(b_m,b_{m+1})$-branch are $2$-path and at most one of them is of length more than one.
\item
For every $i$, $2\leq i\leq k-1$, at most one of the $(b_{i-1},b_i)$-branches and $(b_{i},b_{i+1})$-branches is a cane.
\item
All $(a_i,b_i)$-branches, $(a_{i},b_{i+1})$-branches and $(a_{i},b_{i-1})$-branches are $2$-path.
\end{enumerate}

\begin{center}
\begin{tikzpicture}
[inner sep=0.5mm, place/.style={circle,draw=black,fill=black,thick}]
\node[place] (a1) at (-2,1.5) [label=above:$a_1$] {};
\node[place] (a2) at (-1,1.5) [label=above:$a_2$] {}edge [-,thick](a1);
\node[place] (a3) at (0,1.5) [label=above:$a_3$] {}edge [-,thick](a2);
\node[place] (ak-1) at (1,1.5) [label=above:$a_{k-1}$] {};
\node[place] (ak) at (2,1.5) [label=above:$a_k$] {}edge [-,thick](ak-1);
\node[place] (b1) at (-2,.5) [label=below:$b_1$] {}edge [-,thick](a1);
\node[place] (b2) at (-1,.5) [label=below:$b_2$] {}edge [-,thick](a2)edge [-,thick](b1)edge [-,thick](a1);
\node[place] (b3) at (0,.5) [label=below:$b_3$] {}edge [-,thick](a3)edge [-,thick](b2)edge [-,thick](a2);
\node[place] (bk-1) at (1,.5) [label=below:$b_{k-1}$] {}edge [-,thick](ak-1);
\node[place] (bk) at (2,.5) [label=below:$b_k$] {}edge [-,thick](ak)edge [-,thick](bk-1)edge [-,thick](ak-1);
\node (dots) at (.5,1) [label=center:$\cdots$]{};
\node (a) at (0,-.5) [label=center:$(a)$]{};
\node[place] (a1') at (-4.5,-2) [label=above:$a_1$] {};
\node[place] (a2') at (-3.5,-2) [label=above:$a_2$] {}edge [-,thick](a1');
\node[place] (a3') at (-2.5,-2) [label=above:$a_3$] {}edge [-,thick](a2');
\node[place] (am-2) at (-1.5,-2) [label=above:$a_{m-2}$] {};
\node[place] (am-1) at (-.5,-2) [label=above:$a_{m-1}$] {}edge [-,thick](am-2);

\node[place] (b1') at (-4.5,-3) [label=below:$b_1$] {}edge [-,thick](a1')edge [-,thick](a2');
\node[place] (b2') at (-3.5,-3) [label=below:$b_2$] {}edge [-,thick](a2')edge [-,thick](b1')edge [-,thick](a3');
\node[place] (b3') at (-2.5,-3) [label=below:$b_3$] {}edge [-,thick](a3')edge [-,thick](b2');
\node[place] (bm-2) at (-1.5,-3) [label=below:$b_{m-2}$] {}edge [-,thick](am-1)edge [-,thick](am-2);
\node[place] (bm-1) at (-.5,-3) [label=below:$b_{m-1}$] {}edge [-,thick](am-1)edge [-,thick](bm-2);
\node (dots'') at (-2,-2.5) [label=center:$\cdots$]{};

\node[place] (am) at (.5,-2) [label=above:$a_m$] {}edge [-,thick](bm-1)edge [-,thick](am-1);
\node[place] (am+1) at (1.5,-2) [label=above:$a_{m+1}$] {}edge [-,thick](am);
\node[place] (am+2) at (2.5,-2) [label=above:$a_{m+2}$] {}edge [-,thick](am+1);
\node[place] (am-1) at (3.5,-2) [label=above:$a_{k-1}$] {};
\node[place] (ak') at (4.5,-2) [label=above:$a_k$] {}edge [-,thick](am-1);
\node[place] (bm) at (.5,-3) [label=below:$b_m$] {}edge [-,thick](am)edge [-,thick](bm-1);
\node[place] (bm+1) at (1.5,-3) [label=below:$b_{m+1}$] {}edge [-,thick](am+1)edge [-,thick](bm)edge [-,thick](am);
\node[place] (bm+2) at (2.5,-3) [label=below:$b_{m+2}$] {}edge [-,thick](am+2)edge [-,thick](bm+1)edge [-,thick](am+1);
\node[place] (bm-1) at (3.5,-3) [label=below:$b_{k-1}$] {}edge [-,thick](am-1);
\node[place] (bk') at (4.5,-3) [label=below:$b_k$] {}edge [-,thick](ak')edge [-,thick](bm-1)edge [-,thick](am-1);
\node (dots') at (3,-2.5) [label=center:$\cdots$]{};
\node (b) at (0,-4) [label=center:$(b)$]{};
\end{tikzpicture}
\captionof{figure}{Two different forms of $G_0$.\label{fig:2-path P}}
\end{center}

\begin{thm}\label{thm:Raft}
If $G\in {\cal F}$, then $\dim_M(G)=2$.
\end{thm}
\begin{proof}{
Let $G\in {\cal F}$. Through the proof all of notations are the same as those which are used to introduce the family $\cal F$ and $G_0$ in Figure \ref{fig:2-path P}.
Since $G$ is not a path, $\dim_M(G)\geq 2$.
Let $W=\{a_1,a_k\}$. We  show  in both possible cases for $G_0$ that $W$ is a resolving set for $G$ and hence, $\dim_M(G)=2$.

{\textbf{Case 1.} } $G_0$ is a $2$-path as shown in Figure \ref{fig:2-path P}(a).\\
The metric representation of the vertices $\{a_1,a_2,\ldots,a_k,b_1,b_2,\ldots,b_k\}$ are as follows.
\begin{align*}
&r(a_i|W)=(i-1,k-i),1\leq i\leq k,\\
&r(b_1|W)=(1,k),\\
&r(b_j|W)=(j-1,k-j+1), 2\leq j\leq k.
\end{align*}

Thus, different vertices of $G_0$ have different metric representations.
Moreover, note that
$$\{d_1-d_2:~(d_1,d_2)=r(a_i|W), ~1\leq i\leq k\}=\{1-k,3-k,5-k,\ldots,2i-k-1,\ldots,k-3,k-1\},$$
and
$$\{d_1-d_2:~(d_1,d_2)=r(b_i|W), ~1\leq i\leq k\}=\{1-k,2-k,4-k,\ldots,2i-k-2,\ldots,k-4,k-2\}.$$
If $G=G_0$, then we are done. Suppose that $G\neq G_0$ and let $H$ be a branch of $G$ on an edge $e$ of $G_0$. 
Regards to the  structures of graphs in  ${\cal F}$, we consider the following different possibilities.
\begin{itemize}
\item $H$ is a branch on the vertical edge $e=a_ib_i$, $2\leq i\leq k-1$.\\
Note that by the definition of ${\cal F}$, $H$ is a $2$-path and $\deg_H(a_i)=2$. Let $V(H)=\{x_1,x_2,\ldots,x_t\}$ where $x_1=a_i$, $x_2=b_i$, and $E(H)=\{x_rx_s:~|r-s|\leq 2\}$. If $j$ is odd, then $d(x_j,a_1)=d(x_j,a_i)+d(a_i,a_1)$ and $d(x_j,a_k)=d(x_j,a_i)+d(a_i,a_k)$. If $j$ is even, then $d(x_j,a_1)=d(x_j,b_i)+d(b_i,a_1)$ and $d(x_j,a_k)=d(x_j,b_i)+d(b_i,a_k)$. Hence, we have
\begin{eqnarray*}
r(x_j|W)=
\left\{ \begin{array}{ll} (i-1+\lfloor {j\over 2}\rfloor,k-i+\lfloor {j\over 2}\rfloor) & ~~j ~\mbox{is odd} \\[.2cm]
(i-1+\lfloor {j\over 2}\rfloor-1,k-i+\lfloor {j\over 2}\rfloor) & ~~j ~\mbox{is even}. \end{array} \right.
\end{eqnarray*}
Moreover, note that
$$\{d_1-d_2: ~(d_1,d_2)=r(x_j|W), ~1\leq j\leq t\}=\{2i-k-1,2i-k-2\}.$$
\item $H$ is a branch on the oblique edge $e=a_ib_{i+1}$, $2\leq i\leq k-1$.\\
By the definition of ${\cal F}$, $H$ is a $2$-path and $\deg_H(a_i)=2$. Let $V(H)=\{x_1,x_2,\ldots,x_t\}$ where $x_1=a_i$, $x_2=b_{i+1}$, and $E(H)=\{x_rx_s:~|r-s|\leq 2\}$. If $j$ is odd, then $d(x_j,a_1)=d(x_j,a_i)+d(a_i,a_1)$ and $d(x_j,a_k)=d(x_j,a_i)+d(a_i,a_k)$. If $j$ is even, then $d(x_j,a_1)=d(x_j,b_{i+1})+d(b_{i+1},a_1)$ and $d(x_j,a_k)=d(x_j,b_{i+1})+d(b_{i+1},a_k)$. Hence, we have
\begin{eqnarray*}
r(x_j|W)=
\left\{ \begin{array}{ll} (i-1+\lfloor {j\over 2}\rfloor,k-i+\lfloor {j\over 2}\rfloor) & j ~\mbox{is odd} \\[.2cm]
(i-1+\lfloor {j\over 2}\rfloor,k-i+\lfloor {j\over 2}\rfloor-1) & j ~\mbox{is even} . \end{array} \right.
\end{eqnarray*}
Moreover, note that
$$\{d_1-d_2: ~(d_1,d_2)=r(x_j|W), ~1\leq j\leq t\}=\{2i-k-1,2i-k\}.$$
\item $H$ is a branch on the horizontal edge $e=b_ib_{i+1}$, $1\leq i\leq k-1$. \\
Using the definition of ${\cal F}$, $H$ is either a $2$-path or a cane. Generally, assume that
$$\{x_1,x_2,\ldots,x_t\}\subseteq V(H)\subseteq \{x_1,x_2,\ldots,x_t\}\cup\{x\},$$
where the induced subgraph of $H$ on $\{x_1,x_2,\ldots,x_t\}$ is a $2$-path with the edge set $\{x_rx_s:~|r-s|\leq 2\}$.
We consider two different possibilities.
\begin{itemize}
\item[a)] $x_1=b_i$, $x_2=b_{i+1}$. Hence, if  $H$ is a cane,  then we have $N_H(x)=\{b_i,x_3\}$. Similar to the previous cases, we have
\begin{align*}
&r(x_1|W)=(i-1,k-i+1),\\
&r(x_j|W)=
\left\{ \begin{array}{ll} (i-1+\lfloor {j\over 2}\rfloor,k-i+\lfloor {j\over 2}\rfloor) & j\geq3~ \mbox{is odd} \\[.2cm]
(i-1+\lfloor {j\over 2}\rfloor,k-i+\lfloor {j\over 2}\rfloor-1) & j ~\mbox{is even} . \end{array} \right.
\end{align*}
Also, if $H$ is a cane, then $r(x|W)=(i-1+1,k-i+2)$.\\
\item[b)] $x_1=b_{i+1}$, $x_2=b_i$. Hence, if $H$ is a cane, then we have $N_H(x)=\{b_{i+1},x_3\}$. Similarly, we have
\begin{align*}
&r(x_1|W)=(i-1+1,k-i),\\	
&r(x_j|W)=
\left\{ \begin{array}{ll} (i-1+\lfloor {j\over 2}\rfloor,k-i+\lfloor {j\over 2}\rfloor) &  j ~\mbox{is odd} \\[.2cm]
(i-1+\lfloor {j\over 2}\rfloor-1,k-i+\lfloor {j\over 2}\rfloor) & j ~\mbox{is even} . \end{array} \right.
\end{align*}
Also, if $H$ is a cane, then $r(x|W)=(i-1+2,k-i+1)$. \\ 
\end{itemize}
Note that in both states (and regardless of being a $2$-path or a cane), we have
$$\{d_1-d_2: ~(d_1,d_2)=r(v|W), ~v\in V(H)\}=\{2i-k-2,2i-k-1,2i-k\}.$$
\end{itemize}

Therefore, in all the above cases,  distinct vertices of $H$ have different metric representations. Also,
the metric representation of the vertices in $V(H)$ are different from the metric representations of the vertices in $V(G_0)\setminus \{x,y\}$, where $H$ is a $(x,y)$-branch.
Moreover, using the subtraction value of two coordinates in the metric representation of each vertex, it is easy to check  that vertices of different (possible) branches on $G_0$ (satisfying the conditions mentioned in the definition of ${\cal F}$) have different metric representations. Thus, in this case $W$ is a resolving set for $G$.


{\textbf{Case 2.}}  $G_0$ is a  $2$-tree of the form Figure \ref{fig:2-path P}(b).\\
The metric representation of the vertices $\{a_1,a_2,\ldots,a_m,\ldots,a_k\}\cup\{b_1,b_2,\ldots,b_m,\ldots,b_k\}$ are as follows.
\begin{align*}
&r(a_i|W)=(i-1,k-i),1\leq i\leq k,\\
&r(b_j|W)=
\left\{
\begin{array}{ll}   (j,k-j) & 1\leq j\leq m-1 \\  (m,k-m+1) & j=m  \\  (j-1,k-j+1) & m+1\leq j\leq k. \end{array}  \right.
\end{align*}

Therefore, different vertices of $G_0$ have different metric representations.
Moreover, note that
\begin{eqnarray*}
\{d_1-d_2: ~(d_1,d_2)=r(a_i|W), ~1\leq i\leq k\}=~~~~~~~~~~~~~~~~~~~\\
\{1-k,3-k,5-k,\ldots,2m-k-3,2m-k-1,2m-k+1,\ldots,k-3,k-1\},
\end{eqnarray*}
and
\begin{eqnarray*}
\{d_1-d_2:~(d_1,d_2)=r(b_j|W), ~1\leq j\leq k\}=~~~~~~~~~~~~~~~~\\
\{2-k,4-k,6-k,\ldots,2m-k-2,2m-k-1,2m-k,\ldots,k-4,k-2\}.
\end{eqnarray*}
If $G=G_0$, then we are done. Hence, suppose that $G\neq G_0$ and let $H$ be a branch of $G$ on an edge $e$ of $G_0$. Again, using the possible structures of $H$ according to the definition of ${\cal F}$, we consider the following different cases.
\begin{itemize}
\item $H$ is a branch on the vertical edge $e=a_ib_i$, $2\leq i\leq m-1$.\\
Note that by the definition of ${\cal F}$, $H$ is a $2$-path and $\deg_H(a_i)=2$. Let $V(H)=\{x_1,x_2,\ldots,x_t\}$ where $x_1=a_i$, $x_2=b_i$, and $E(H)=\{x_rx_s:~|r-s|\leq 2\}$.
It is straightforward to check that
\begin{eqnarray*}
r(x_j|W)=
\left\{ \begin{array}{ll} (i-1+\lfloor {j\over 2}\rfloor,k-i+\lfloor {j\over 2}\rfloor) & j ~\mbox{is odd} \\[.2cm]
(i+\lfloor {j\over 2}\rfloor-1,k-i+\lfloor {j\over 2}\rfloor-1) & j ~\mbox{is even}. \end{array} \right.
\end{eqnarray*}
Moreover, note that
$$\{d_1-d_2: ~(d_1,d_2)=r(x_j|W), ~1\leq j\leq t\}=\{2i-k-1,2i-k\}.$$
\item $H$ is a branch on the vertical edge $e=a_ib_i$, $m+1\leq i\leq k-1$.\\
By the definition of ${\cal F}$, $H$ is a $2$-path and $\deg_H(a_i)=2$. Let $V(H)=\{x_1,x_2,\ldots,x_t\}$ where $x_1=a_i$, $x_2=b_i$, and $E(H)=\{x_rx_s:~|r-s|\leq 2\}$.
We have
\begin{eqnarray*}
r(x_j|W)=
\left\{ \begin{array}{ll} (i-1+\lfloor {j\over 2}\rfloor,k-i+\lfloor {j\over 2}\rfloor) & j ~\mbox{is odd} \\[.2cm]
(i+\lfloor {j\over 2}\rfloor-2,k-i+\lfloor {j\over 2}\rfloor) & j ~\mbox{is even} . \end{array} \right.
\end{eqnarray*}
Moreover, note that
$$\{d_1-d_2: ~(d_1,d_2)=r(x_j|W), ~1\leq j\leq t\}=\{2i-k-1,2i-k-2\}.$$
\item $H$ is a branch on the oblique edge $e=a_ib_{i-1}$, $2\leq i\leq m-1$.\\
Since $G\in {\cal F}$, $H$ is a $2$-path and $\deg_H(a_i)=2$. Let $V(H)=\{x_1,x_2,\ldots,x_t\}$ where $x_1=a_i$, $x_2=b_{i-1}$, and $E(H)=\{x_rx_s:~|r-s|\leq 2\}$.
We have
\begin{eqnarray*}
r(x_j|W)=
\left\{ \begin{array}{ll} (i-1+\lfloor {j\over 2}\rfloor,k-i+\lfloor {j\over 2}\rfloor) & j ~\mbox{is odd} \\[.2cm]
(i+\lfloor {j\over 2}\rfloor-2,k-i+\lfloor {j\over 2}\rfloor) & j ~\mbox{is even} . \end{array} \right.
\end{eqnarray*}
Moreover, 
$$\{d_1-d_2: ~(d_1,d_2)=r(x_j|W), ~1\leq j\leq t\}=\{2i-k-1,2i-k-2\}.$$
\item $H$ is a branch on the oblique edge $e=a_ib_{i+1}$, $m+1\leq i\leq k-1$.\\
 We know that $H$ is a $2$-path and $\deg_H(a_i)=2$. Let $V(H)=\{x_1,x_2,\ldots,x_t\}$ where $x_1=a_i$, $x_2=b_{i+1}$, and $E(H)=\{x_rx_s:~|r-s|\leq 2\}$.
Similarly, it can be easily checked that 
\begin{eqnarray*}
r(x_j|W)=
\left\{ \begin{array}{ll} (i-1+\lfloor {j\over 2}\rfloor,k-i+\lfloor {j\over 2}\rfloor) & j ~\mbox{is odd} \\[.2cm]
(i+\lfloor {j\over 2}\rfloor-1,k-i+\lfloor {j\over 2}\rfloor-1) & j ~\mbox{is even} . \end{array} \right.
\end{eqnarray*}
Moreover, note that
$$\{d_1-d_2: ~(d_1,d_2)=r(x_j|W), ~1\leq j\leq t\}=\{2i-k-1,2i-k\}.$$
\item $H$ is a branch on the horizontal edge $e=b_ib_{i+1}$, $1\leq i\leq m-2$. \\
Using the definition of ${\cal F}$, $H$ is either a $2$-path or a cane. Generally, assume that
$$\{x_1,x_2,\ldots,x_t\}\subseteq V(H)\subseteq \{x_1,x_2,\ldots,x_t\}\cup\{x\},$$
where the induced subgraph of $H$ on $\{x_1,x_2,\ldots,x_t\}$ is a $2$-path with the edge set $\{x_rx_s:~|r-s|\leq 2\}$.
We consider two different possibilities.
\begin{itemize}
\item[a)] $x_1=b_i$, $x_2=b_{i+1}$.  Hence, if  $H$ is a cane,  then we have $N_H(x)=\{b_i,x_3\}$. Similar to the previous cases, we have
\begin{align*}
&r(x_1|W)=(i,k-i),\\
&r(x_j|W)=
\left\{ \begin{array}{ll} (i+\lfloor {j\over 2}\rfloor,k-i+\lfloor {j\over 2}\rfloor-1) &  j\geq 3~\mbox{is odd} \\[.2cm]
(i+\lfloor {j\over 2}\rfloor,k-i+\lfloor {j\over 2}\rfloor-2) & j ~\mbox{is even} . \end{array} \right.
\end{align*}
Also, if $H$ is a cane, then $r(x|W)=(i+1,k-i+1)$.\\
\item[b)] $x_1=b_{i+1}$, $x_2=b_i$.  Hence,  if  $H$ is a cane,  then we have $N_H(x)=\{b_{i+1},x_3\}$. Similarly, we have
\begin{align*}
&r(x_1|W)=(i+1,k-i-1),\\
&r(x_j|W)=
\left\{ \begin{array}{ll} (i+\lfloor {j\over 2}\rfloor,k-i+\lfloor {j\over 2}\rfloor-1) & j\geq 3~\mbox{is odd} \\[.2cm]
(i+\lfloor {j\over 2}\rfloor-1,k-i+\lfloor {j\over 2}\rfloor-1) & \mbox{is even}. \end{array} \right.
\end{align*}
Also, if $H$ is a cane, then $r(x|W)=(i+2,k-i)$.\\ 
\end{itemize}
Note that in the both states (and regardless of being a $2$-path or a cane) we have
$$\{d_1-d_2: ~(d_1,d_2)=r(v|W),  v\in V(H)\}=\{2i-k,2i-k+1,2i-k+2\}.$$
\item $H$ is a branch on the horizontal edge $e=b_{m-1}b_m$. \\
By the definition of ${\cal F}$, $H$ is a $2$-path and $\deg_H(b_{m-1})=2$. Let $V(H)=\{x_1,x_2,\ldots,x_t\}$ where $x_1=b_{m-1}$, $x_2=b_m$, and $E(H)=\{x_rx_s:~|r-s|\leq 2\}$.
We have
\begin{eqnarray*}
r(x_j|W)=
\left\{ \begin{array}{ll} (m+\lfloor {j\over 2}\rfloor-1,k-m+\lfloor {j\over 2}\rfloor+1) & j ~\mbox{is odd} \\[.2cm]
(m+\lfloor {j\over 2}\rfloor-1,k-m+\lfloor {j\over 2}\rfloor) & j ~\mbox{is even}. \end{array} \right.
\end{eqnarray*}
Moreover, note that
$$\{d_1-d_2: ~(d_1,d_2)=r(x_j|W), ~1\leq j\leq t\}=\{2m-k-2,2m-k-1\}.$$
\item $H$ is a branch on the horizontal edge $e=b_mb_{m+1}$. \\
By the definition of ${\cal F}$, $H$ is a $2$-path and $\deg_H(b_{m+1})=2$. Let $V(H)=\{x_1,x_2,\ldots,x_t\}$ where $x_1=b_{m+1}$, $x_2=b_m$, and $E(H)=\{x_rx_s:~|r-s|\leq 2\}$.
We have
\begin{eqnarray*}
r(x_j|W)=
\left\{ \begin{array}{ll} (m+\lfloor {j\over 2}\rfloor,k-m+\lfloor {j\over 2}\rfloor) & j ~\mbox{is odd} \\[.2cm]
(m+\lfloor {j\over 2}\rfloor-1,k-m+\lfloor {j\over 2}\rfloor) & ~j ~~\mbox{even} . \end{array} \right.
\end{eqnarray*}
Moreover, note that
$$\{d_1-d_2: ~(d_1,d_2)=r(x_j|W), ~1\leq j\leq t\}=\{2m-k-1,2m-k\}.$$
\item $H$ is a branch on the horizontal edge $e=b_ib_{i+1}$, $m+1\leq i\leq k-1$. \\
Using the definition of ${\cal F}$, $H$ is either a $2$-path or a cane. Generally, assume that
$$\{x_1,x_2,\ldots,x_t\}\subseteq V(H)\subseteq \{x_1,x_2,\ldots,x_t\}\cup\{x\},$$
where the induced subgraph of $H$ on $\{x_1,x_2,\ldots,x_t\}$ is a $2$-path with the edge set $\{x_rx_s:~|r-s|\leq 2\}$.
Again, we consider two different possibilities.
\begin{itemize}
\item[a)] $x_1=b_i$, $x_2=b_{i+1}$. Hence,  if  $H$ is a cane and  $N_H(x)=\{b_i,x_3\}$, then  We have
\begin{align*}
&r(x_1|W)=(i-1,k-i+1),\\
&r(x_j|W)=
\left\{ \begin{array}{ll} (i+\lfloor {j\over 2}\rfloor-1,k-i+\lfloor {j\over 2}\rfloor) & j\geq 3~ \mbox{is odd} \\[.2cm]
(i+\lfloor {j\over 2}\rfloor-1,k-i+\lfloor {j\over 2}\rfloor-1) & j ~\mbox{is even} . \end{array} \right.
\end{align*}
Also, if $H$ is a cane, then $r(x|W)=(i,k-i+2)$. \\
\item[b)] $x_1=b_{i+1}$, $x_2=b_i$.  Hence, if  $H$ is a cane,  then we have $N_H(x)=\{b_{i+1},x_3\}$. Similarly, we have
\begin{align*}
&r(x_1|W)=(i,k-i),\\
&r(x_j|W)=
\left\{ \begin{array}{ll} (i+\lfloor {j\over 2}\rfloor-1,k-i+\lfloor {j\over 2}\rfloor) & j\geq3~ \mbox{is odd} \\[.2cm]
(i+\lfloor {j\over 2}\rfloor-2,k-i+\lfloor {j\over 2}\rfloor) & j ~\mbox{is even} . \end{array} \right.
\end{align*}
Also, if $H$ is a cane, then $r(x|W)=(i+1,k-i+1)$.\\
\end{itemize}
Note that in the both  states (and regardless of being a $2$-path or a cane) we have
$$\{d_1-d_2:~(d_1,d_2)=r(v|W),  v\in V(H)\}=\{2i-k-2,2i-k-1,2i-k\}.$$
\end{itemize}
Therefore, in all of above cases, distinct vertices of $H$ have different metric representations. Also,
the metric representation of the vertices in $V(H)$ are different from the metric representations of the vertices in $V(G_0)\setminus \{x,y\}$, where $H$ is a $(x,y)$-branch.  Moreover, using the subtraction value of two coordinates in the metric representation of each vertex, it is easy to check that vertices of different (possible) branches on $G_0$ (satisfying the conditions mentioned in the definition of ${\cal F}$) have different metric representations. Thus, in this case  $W$ is a resolving set for $G$.
}\end{proof}
To prove the converse of Theorem \ref{thm:Raft}, we need the following lemma.
\begin{lemma}\label{lem:}
Let $H$ be a $\{u,v\}$-branch of $G$ and let $\{a,b\}$ be a basis for $G\cup H$. If $\{a,b\}\cap V(H)\subseteq \{u,v\}$, then $\{u,v\}$ is a metric basis for $H$.
\end{lemma}
\begin{proof}{
 Suppose on the contrary, there are two different vertices $x$ and $y$ in $H$ such that
 $$d(x,u)=d(y,u)=r,~~d(x,v)=d(y,v)=s.$$
 Since $H$ is a branch on $\{u,v\}$,  each path connecting a vertex in $H$ with a vertex in $V(G)\setminus V(H)$  passes through $u$ or $v$. Assume that
 $$d(u,a)=r_1,~~d(v,a)=s_1,~~d(u,b)=r_2,~~d(v,b)=s_2.$$
 Hence,
 $$d(x,a)=\min\{r+r_1, s+s_1\}=d(y,a),~~d(x,b)=\min\{r+r_2,s+s_2\}=d(y,b).$$
 This  contradicts that  $\{a,b\}$ is a resolving set for $G\cup H$.
 }\end{proof}

Now, we prove that every $2$-dimensional $2$-tree belongs to the family $\cal F$.

\begin{thm}
If $G$ is a $2$-tree of metric dimension two, then $G\in {\cal F}$.
\end{thm}
\begin{proof}{
Let  $G$ be a $2$-tree and $\{a,b\}$ be a basis of $G$. If $d(a,b)=1$, then by Proposition~\ref{lem:d=1}, $G$ is a $2$-path or a cane which belongs to ${\cal F}$. Thus, assume that $d(a,b) >1$ and let $H$ be  a minimal induced $2$-connected  subgraph of $G$ as shown in Figure \ref{myfig}, containing $a$ and $b$.
Since the clique number of $G$ is three, in each square exactly one of the dashed edges are allowed. Moreover,
by the minimality of $H$ we have $deg_H(a)=deg_H(b)=2$, where $a\in \{a_1, b_1\}$ and $b\in \{a_k,b_k\}$.
Hence, one of two vertices $a_1, b_1$ or one of two vertices $a_k, b_k$ may not exist.
One can check that $\{a,b\}\neq \{a_1, b_k\}$  and  $\{a,b\}\neq \{b_1, a_k\}$, otherwise, two neighbours of $a$ or $b$ get the same metric representation. Thus, by the symmetry,  we may assume  $\{a,b\}=\{a_1, a_k\}$.

\begin{center}
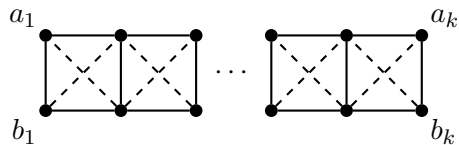

\begin{tikzpicture}
[inner sep=0.5mm, place/.style={circle,draw=black,fill=black,thick}]
\node[place] (a1) at (-2,1) [label=above left:$a_1$] {};
\node[place] (a2) at (-1,1)             
 {}edge [-,thick](a1);
\node[place] (a3) at (0,1)          
 {}edge [-,thick](a2);
\node[place] (a4) at (1,1)          
 {};
\node[place] (ak-1) at (2,1)        
 {}edge [-,thick](a4);
\node[place] (ak) at (3,1) [label=above right:$a_k$] {}edge [-,thick](ak-1);
\node[place] (b1) at (-2,0) [label=below left:$b_1$] {}edge [-,thick](a1)edge [-,thick,dashed](a2);
\node[place] (b2) at (-1,0)         
{}edge [-,thick](a2)edge [-,thick](b1)edge [-,thick,dashed](a1)edge [-,thick,dashed](a3);
\node[place] (b3) at (0,0)          
{}edge [-,thick](a3)edge [-,thick](b2)edge [-,thick,dashed](a2);
\node[place] (b4) at (1,0)          
 {}edge [-,thick,dashed](ak-1)edge [-,thick](a4);
\node[place] (bk-1) at (2,0)        
 {}edge [-,thick](ak-1)edge [-,thick,dashed](ak)edge [-,thick,dashed](a4)edge [-,thick](b4);
\node[place] (bk) at (3,0) [label=below right:$b_k$] {}edge [-,thick](ak)edge [-,thick](bk-1)edge [-,thick,dashed](ak-1);
\node (dots) at (.5,.5) [label=center:$\cdots$]{};
\end{tikzpicture}
\captionof{figure}{A minimal induced $2$-connected  subgraph of $G$.
\label{myfig}}
\end{center}

If $\Delta(H)\leq4$, then $H$ is a $2$-path as shown in Figure~\ref{fig:2-path P}(a).  Otherwise $\Delta(H)=5$. If there exists a vertex $b_j$ of degree $5$, then it  can be easily checked that $b_j$ and $a_j$ have the same representation with respect to $\{a_1, a_k\}$. 
Also, existence of two vertices $a_i$ and $a_{i'}$ both of degree $5$, $i\leq i'$, implies that there exists some vertex $b_j$, $i\leq j\leq i'$, of degree $5$, which is impossible. Thus, there exists a unique $a_i$ of degree $5$.
Therefore,  $H$ is the graph shown in Figure \ref{fig:2-path P}(b). Thus, $H$ is a $2$-path or a $2$-tree obtained by identifying the specific edge, say $a_mb_m$, of two $2$-paths (see Figure~\ref{fig:2-path P}(b)), where $B=\{a_1,a_k\}$.
Thus, $G$ satisfies property (1).

Clearly, on every edge there is at most one branch; thus, property (2) follows.  Also, $G$ avoids any $(a_i,a_{i+1})$-branch, because each vertex adjacent to both $a_i$ and $a_{i+1}$ has the same metric representation as $b_{i}$ or $b_{i+1}$.
Thus, $G$ contains only $(a_i,b_i)$-branches, $(a_i,b_{i+1})$-branches, $(a_{i+1},b_{i})$-branches or  $(b_i,b_{i+1})$-branches; which implies property (3). 
Moreover, by Proposition \ref{lem:d=1} and Lemma \ref{lem:}, each of these branches is
a $2$-path or a cane. Therefore, property (4) holds. 
Also, by Theorem \ref{thm:degree of basis elements}, for every $i$, $1\leq i\leq k$, there is at most one $(a_i,x)$-branch in $G$. Moreover, in each $(a_i,b_i)$-branch the degree of $a_i$ is two, which shows trueness of property (5).

To see property (6), first note that by property (3) there is no $(a_{m-1}, a_m)$-branch or $(a_m, a_{m+1})$-branch. Moreover, in each $(a_m, x)$-branch, for $x\in\{b_{m-1}, b_m, b_{m+1}\}$, the unique neighbour of $a_m$ on the branch has the same metric representation as $b_m$.  

To show that $G$ has property (7), suppose that a triangle $a_ib_ib_{i+1}$ has more than one branch. By Theorem \ref{thm:degree of basis elements}, at most one of  $(a_i,b_i)$-branch and $(a_i,b_{i+1})$-branch exists. Therefore, $b_ib_{i+1}$ has a branch $H_1$ and one of the edges $a_ib_i$ or $a_ib_{i+1}$  has another branch $H_2$. Let $x$ and $y$ be the vertices of distance one from $G_0$ on branches $H_1$ and $H_2$, respectively. Hence, $d(a_1,x)=d(a_1,y)=i$ and $d(a_k,x)=d(a_k,y)=k-i+1$. That is, $\{a_1,a_k\}$ is not a basis of $G$, which is a contradiction. A similar reason works for triangle $a_ib_{i-1}b_i$. Hence, $G$ has property (7).

Let $(d_1,d_2)$ be metric representation of $b_i$. Then metric representations of each neighbour of $b_i$ which is out of $G_0$ could be one of $(d_1+1,d_2+1),~(d_1+1,d_2)$ or $(d_1,d_2+1)$. Thus,  $b_i$ has at most three neighbours out of $G_0$. Hence, the degree of $b_i$ in $G$ is at most $7$ that is property (8).

If there are two branches of length at least $2$ on a triangle containing $a_ia_{i+1}$, then the metric representation of the second vertices on these branches are the same, a contradiction. Thus, $G$ satisfies property (9).

If $H$ is a $(b_{m-1},b_m)$-branch of cane type, then one can find two vertices in $N_G(b_m)\cup N_G(b_{m-1})$ with the same metric representation. A similar argument holds whenever $H$ is  a $(b_{m},b_{m+1})$-branch of cane type.
If there is a $(b_{m-1},b_{m})$-branch, say  $H_1$, and a $(b_{m},b_{m+1})$-branch, say $H_2$, both of  length at least two, then $b_m$ has a neighbour in $H_1$ with the same metric representation as a  neighbour of $b_m$ in $H_2$. Hence, property (10) holds. 

Suppose that two branches on $(b_{i-1},b_i)$ and $(b_{i},b_{i+1})$  are canes. In this case, it can be checked that in the set of neighbours of $b_i$ in these branches there are two vertices with the same metric representation. Thus, $G$ satisfies property (11).

Using Theorem 1.1 the degree of each $a_i$ in $G$, $1<i<n$, is at most five. Note that $\deg(a_i)\in\{4,5\}$. Now suppose that $H$ is a branch on the edge $\{a_i,b_i\}$, $\{a_i,b_{i+1}\}$ or $\{a_i,b_{i-1}\}$. If $H$ is a cane, then $\deg_G(a_i)\geq 6$ or two neighbours of $b_{i-1}$, $b_{i}$ or $b_{i+1}$ in $H$ get the same metric representation, which both are contradictions. Thus, each branch on the edge $\{a_i,b_{i-1}\}$, $\{a_i,b_i\}$ or $\{a_i,b_{i+1}\}$ is a 2-path and $G$ satisfies property (12).
}\end{proof}


%


\end{document}